\documentclass[12pt, psamsfonts]{amsart}
\usepackage{amssymb,amsfonts,amsthm,amsmath,epsfig,hhline}
\usepackage{verbatim, enumerate}    

\usepackage[arc,all,graph,frame]{xy}

\theoremstyle{plain}
\newtheorem{theorem}{Theorem}

\theoremstyle{definition}
\newtheorem{definition}{Definition}

\newtheorem{remark}{Remark}

\def\d{\underbar{\em d}}

\title{Symmetric Bipartite Graphs and  Graphs with Loops} 

\author{Grant Cairns}
\author{Stacey Mendan}

\address{Department of Mathematics and Statistics, La Trobe University, Melbourne, Australia 3086}
\email{G.Cairns@latrobe.edu.au}
\email{spmendan@students.latrobe.edu.au}

\begin{document}

\maketitle

\begin{abstract}
We show that if the two parts of a finite bipartite graph have the same degree sequence, then there is a  bipartite graph, with the same degree sequences, which is symmetric, in that it has  an involutive graph automorphism that interchanges its two parts.
To prove this, we study the relationship between symmetric bipartite graphs and graphs with loops.
\end{abstract}

\section{Introduction}

We say that a finite sequence $\d$ of nonnegative  integers is \emph{bipartite graphic} if the pair $(\d,\d\,)$ can be realized as the degree sequences of the parts of a bipartite simple graph. For example, Figure \ref{Fex} gives two realizations of the sequence $(2,2,1,1)$. 
Notice that the realization on the left is symmetric, while the one on the right is not, where by \emph{symmetric} we use the following natural definition.

\begin{definition} We say that a bipartite graph $G$ is \emph{symmetric} if there is an involutive graph automorphism of $G$ that interchanges its two parts.
\end{definition}

\begin{figure}[h!]
\[
\hbox{\xymatrix{
*{\circ}\ar@{-}[d]\ar@{-}[dr]&*{\circ}\ar@{-}[d]\ar@{-}[dl]&*{\circ}\ar@{-}[d]&*{\circ}\ar@{-}[d]\\
*{\circ}&*{\circ}&*{\circ}&*{\circ}}}
\hskip1.5cm
\hskip1cm
\hbox{\xymatrix{
*{\circ}\ar@{-}[d]\ar@{-}[dr]&*{\circ}\ar@{-}[dr]&*{\circ}\ar@{-}[d]\ar@{-}[dr]&*{\circ}\ar@{-}[d]\\
*{\circ}&*{\circ}&*{\circ}&*{\circ}}}
\]
\caption{}\label{Fex}
\end{figure}

We will establish the following result.

\begin{theorem}\label{T:sym} If a sequence $\emph{\d}$ is bipartite graphic, then there is a realization of $\emph{\d}$ that is symmetric.
\end{theorem}

Our proof of Theorem \ref{T:sym} relies on an observation connecting symmetric bipartite graphs with graph-with-loops.

\begin{definition} By a \emph{graph-with-loops} we mean a graph, without multiple edges, in which there is at most one loop at each vertex.
\end{definition}

Given a symmetric bipartite graph $G$ with involution $\sigma$, the quotient graph $G/\sigma$ is clearly a graph-with-loops. To see how this process can be reversed, recall that for every simple graph $G$, there is a natural associated bipartite
simple  graph
$\hat G$ called the \emph{bipartite double-cover} of $G$. The graph  $\hat G$  is the \emph{tensor product} $G\times K_2$ of $G$ with the connected graph $K_2$ with 2 vertices;  the vertex set of $G\times K_2$ is the Cartesian product of the vertices of $G$ and $K_2$, there are edges in $G\times K_2$ between $(a,0)$ and $(b,1)$ and between  $(a,1)$ and $(b,0)$ if and only if there is an edge in $G$ between $a$ and $b$; see \cite{HIK}. By construction, $\hat G$ is bipartite and symmetric; the automorphism is the map $(a,x)\mapsto (a,1-x)$. When $G$ is a graph-with-loops, the above construction again produces a symmetric bipartite graph; each loop in $G$ produces just one edge in  $G\times K_2$. Figure \ref{Fhex} shows the construction for the complete graph-with-loops  $G$ on three vertices. 

\begin{figure}
\[
\hbox{\xymatrix@=.9em{
\\
&&*{\circ}
\ar@{-}[ddll]\ar@{-}[ddrr]\ar@{-}@(ul,ur)\\
\\
*{\circ}\ar@{-}@(l,d)\ar@{-}[rrrr]&&&&*{\circ}\ar@{-}@(r,d)\\
}}
\hskip.5cm
\hbox{\xymatrix{
&\\
}}
\hskip1cm
\hbox{\xymatrix@=.9em{
&&*{\circ}\ar@{-}[dddd]
\ar@{-}[dll]\ar@{-}[drr]\\
*{\circ}\ar@{-}[dd]\ar@{-}[ddrrrr]&&&&*{\circ}\ar@{-}[dd]\ar@{-}[ddllll]\\
\\
*{\circ}\ar@{-}[drr]&&&&*{\circ}\ar@{-}[dll]\\
&&*{\circ}
}}
\]
\caption{The complete graph-with-loops  on three vertices, and its bipartite double-cover.}\label{Fhex}
\end{figure}
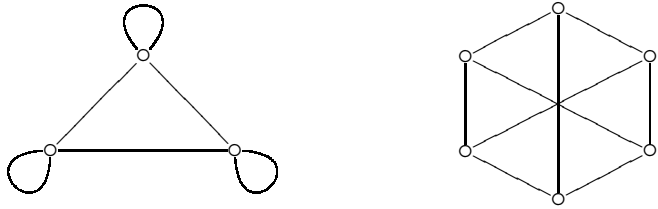

Note that graphs-with-loops are a special family of multigraphs \cite{Diestel} and that for multigraphs, the \emph{degree} of a vertex is usually taken to be the number of edges incident to the vertex, with loops counted twice. For our purposes, a different definition of degree is more appropriate. We introduce the following definition.

\begin{definition} For a graph-with-loops,
 the \emph{reduced degree} of a vertex is  taken to be the number of edges incident to the vertex, with loops counted once. 
\end{definition}

So, for example, in the complete graph-with-loops  $G$ on three vertices, shown on the left in Figure \ref{Fhex}, the vertices each have reduced degree three. The vertices in the tensor product $\hat G=G\times K_2$ also have the same degrees as the reduced degrees of the corresponding vertices of $G$. In general, if $\d$ is the sequence of  reduced degrees of the vertices of a graph-with-loops  $G$, then  $\hat G$  is a symmetric bipartite graph whose parts have degree sequences $(\d,\d\,)$. 

We will employ the following Erd\H{o}s--Gallai type result; the proof is given in the final section.

\begin{theorem}\label{T:reduced} Let $\emph{\d}=(d_1 ,\dots, d_n )$ be a sequence  of nonnegative integers in decreasing order. Then
 $\emph{\d}$ is the sequence of  reduced degrees of the vertices of a graph-with-loops if and only if for each integer $k$ with $1 \leq  k \leq  n$,
\begin{equation}\label{E:loops}
\sum_{i=1}^k d_i\leq  k^2+ \sum_{i=k+1}^n\min\{k,d_i\} .  
\end{equation}
\end{theorem}

 Theorem \ref{T:reduced} has the following application.

\begin{theorem}\label{T:ddd} A  sequence  $\emph{\d}=(d_1 ,\dots, d_n )$ of nonnegative integers in decreasing order is the sequence of  reduced degrees of the vertices of a graph-with-loops if and only if $\emph{\d}$ is bipartite graphic.
\end{theorem}

\begin{proof} If $\d$ is the sequence of  reduced degrees of the vertices of a graph-with-loops $G$, then the bipartite double-cover $\hat G$ of $G$ has parts with degree sequences $(\d,\d\,)$.
Conversely, if $\d$ is bipartite graphic, then by the Gale--Ryser Theorem \cite{Gale,Ryser}, for each $k$ with $1\leq k\leq n$,
\[\sum_{i=1}^k d_i
\leq   \sum_{i=1}^n\min\{k,d_i\}
\leq   k^2+\sum_{i=k+1}^n\min\{k,d_i\},
\]
and so by Theorem \ref{T:reduced}, $\d$ is the sequence of  reduced degrees of the vertices of a graph-with-loops.\end{proof}

We can now prove our main result.

\begin{proof}[Proof of Theorem \ref{T:sym}]
If $\d$ is  bipartite graphic, then by Theorem \ref{T:ddd}, $\d$ is the sequence of  reduced degrees of the vertices of a graph-with-loops $G$. Then the bipartite double-cover  $\hat G$ of $G$ is symmetric and its parts have degree sequences $(\d,\d)$.\end{proof}

\section{Some Remarks}

\begin{remark} From Theorem  \ref{T:reduced} and Theorem \ref{T:ddd}, condition \eqref{E:loops}  gives an Erd\H{o}s--Gallai type condition for a sequence to be bipartite graphic, which is analogous to the Gale--Ryser condition.  \end{remark}

\begin{remark} It is clear from the discussion in Section 1 that if a sequence $(d_1 ,\dots, d_n )$ is graphic, then by adding a loop at each vertex, $(d_1+1 ,d_2+1,\dots, d_n +1)$ is the sequence of  reduced degrees of the vertices of a graph-with-loops, and so by Theorem \ref{T:ddd}, the sequence $(d_1+1 ,d_2+1,\dots, d_n +1)$ is bipartite graphic. Note that the converse is not true; for example, $(4,4,2,2)$ is bipartite graphic, while $(3,3,1,1)$ is not graphic. 
\end{remark}

\begin{remark}
There are several results in the literature of the following kind: if $\d$ is graphic, and if $\d\,'$ is obtained from $\d$ using a particular construction, then $\d\,'$ is also graphic. 
The Kleitman--Wang Theorem is of this kind \cite{KW}. Another useful result is implicit in Choudum's proof \cite{C} of the Erd\H{o}s--Gallai Theorem: If a decreasing sequence $\d=(d_1 ,\dots, d_n )$ of positive integers  is graphic, then so is the sequence $\d\,'$ obtained by reducing both $d_1$ and $d_n$ by one.
Analogously, our proof of Theorem \ref{T:reduced}, which is modelled on Choudum's proof, also establishes the following result:
 If a decreasing sequence $\d=(d_1 ,\dots, d_n )$  of positive integers is bipartite graphic,  then so is the sequence $\d\,'$ obtained by reducing both $d_1$ and $d_n$ by one.
\end{remark}

\begin{remark}
For criteria for sequences to be realized by multigraphs, see \cite{MV}. There are many  other recent papers on graphic sequences, see for example \cite{TV,TT,Yin,YL,BHJW,CSZZ,CSNZZ1,CSNZZ2}.
\end{remark}

\section{Proof of Theorem \ref{T:reduced}}\label{S:EGS}

The following  proof mimics Choudum's proof of the Erd\H{o}s--Gallai Theorem  \cite{C}.

 For the proof of necessity, consider the set $S$ comprised of the first $k$ vertices. The left hand side of \eqref{E:loops} is the number of half-edges incident to $S$, with each loop counting as  one. On the right hand side, $k^2$ is the number of half-edges in the complete graph-with-loops on $S$, again with each loop counting as  one, while $ \sum_{i=k+1}^n\min\{k,d_i\}$ is the maximum number of edges that could join vertices in $S$ to vertices outside $S$. So \eqref{E:loops} is obvious.

Conversely, suppose that $\d=(d_1 ,\dots, d_n )$ verifies \eqref{E:loops} and consider the sequence $\d\,'$ obtained by reducing both $d_1$ and $d_n$ by 1.  Let $\d\,''$ denote the sequence obtained by reordering $\d\,'$ so as to be decreasing. 
Suppose that $\d\,''$ satisfies \eqref{E:loops} and hence by the inductive hypothesis, there is a graph-with-loops $G'$ that realizes $\d\,'$. We will show how $\d$ can be realized. Let the  vertices of $G'$ be labelled $v_1,\dots,v_n$.  If there is no edge in $G'$ connecting $v_1$ to $v_n$, then add one; this gives a graph-with-loops $G$  that realizes $\d$.  Similarly, if there is no loop at either $v_1$ or $v_n$,  just add loops at both $v_1$ and $v_n$. So it remains to treat the case where there is an edge in $G'$ connecting  $v_1$ to $v_n$, and at least one of the vertices $v_1,v_n$ has a loop.

Now, for the moment, let us assume there is a loop  in $G'$ at $v_1$. Applying the hypothesis to $\d$, using $k=1$ gives 
\[
d_1\leq  1+\sum_{i=2}^n\min\{1,d_i\}\leq n,
\]
and so $d_1-2< n-1$. Now in $G'$, the degree of $v_1$ is $d_1-1$ and so apart from the loop at $v_1$, there are a further $d_1-2$ edges incident to $v_1$. So in $G'$, there is some vertex $v_i\not= v_1$, for which there is no edge from $v_1$ to $v_i$. 
Note that $d'_i>d'_n$. 
If there is a loop  in $G'$ at $v_n$, or if there is no loop at $v_i$ nor at $v_n$, then there is a vertex $v_j$ such that there is an edge in $G'$ from $v_i$ to $v_j$, but there is no edge from $v_j$ to $v_n$.
Now remove the edge $v_iv_j$, and put in edges from $v_1$ to $v_i$, and from $v_j$ to $v_n$, as in Figure \ref{F1e}. This 
gives a graph-with-loops $G$ that realizes $\d$.
If there is  no loop  in $G'$ at  $v_n$, but there is a loop at $v_i$,  remove the loop at $v_i$, add the edge $v_1v_i$ and add a loop at $v_n$, as in Figure \ref{F2e}.
 
Finally, assume there is no loop  in $G'$ at $v_1$, but there is a loop in $G'$ at $v_n$. So, apart from the loop, there are a further $d_n-2$ edges incident to $v_n$. Since $d_1\geq d_n$, we have $d_1-1> d_n-2$, and so there is a  vertex $v_i$ such that there is an edge in $G'$ from $v_1$ to $v_i$, but there is no edge from $v_i$ to $v_n$.
Note that $d'_i>d'_n$, so as there is a loop  in $G'$ at $v_n$,  there is a vertex $v_j$ such that there is an edge in $G'$ from $v_i$ to $v_j$, but there is no edge from $v_j$ to $v_n$. Now remove the  loop at $v_n$ and the edge $v_iv_j$, and put edges $v_jv_n$ and $v_iv_n$ and add a loop at $v_1$, as in Figure \ref{F3e}. This 
gives a graph-with-loops $G$ that realizes $\d$.

\begin{figure}
\[
\hbox{\xymatrix{
&&*{\circ}\save[]+<0.2cm,0.2cm>*{v_i}\restore
\ar@{..}[dll]\ar@{-}[dd]&&\\
*{\circ}\save[]+<-0.4cm,0.2cm>*{v_1}\restore\ar@{-}[rrrr]\ar@{-}@(ul,ur)&&&&*{\circ}\save[]+<0.2cm,0.2cm>*{v_n}\restore\\
&&*{\circ}\save[]+<0.3cm,-0.2cm>*{v_j}\restore \ar@{..}[urr]&&}}
\hskip.5cm
\hbox{\xymatrix{
&\\
&\to\\
&}}
\hskip1cm
\hbox{\xymatrix{
&&*{\circ}\save[]+<0.2cm,0.2cm>*{v_i}\restore  \ar@{-}[dll]  \ar@{..}[dd]&&\\
*{\circ}\save[]+<-0.4cm,0.2cm>*{v_1}\restore\ar@{-}[rrrr]\ar@{-}@(ul,ur)&&&&*{\circ}\save[]+<0.2cm,0.2cm>*{v_n}\restore\\
&&*{\circ}\save[]+<0.3cm,-0.2cm>*{v_j}\restore \ar@{-}[urr]&&}}
\]
\caption{}\label{F1e}
\end{figure}

\begin{figure}
\[
\hbox{\xymatrix{
&&*{\circ}\save[]+<-0.1cm,-0.3cm>*{v_i}\restore 
\ar@{..}[dll]\ar@{-}@(ul,ur)\\
*{\circ}\save[]+<-0.4cm,0.2cm>*{v_1}\restore\ar@{-}@(ul,ur)\ar@{-}[rrrr]&&&&*{\circ}\save[]+<0.5cm,0.2cm>*{v_n}\restore\ar@{..}@(ul,ur)\\
}}
\hskip.5cm
\hbox{\xymatrix{
&\\
&\to}}
\hskip1cm
\hbox{\xymatrix{
&&*{\circ}\save[]+<-0.1cm,-0.3cm>*{v_i}\restore 
\ar@{-}[dll]\ar@{..}@(ul,ur)\\
*{\circ}\save[]+<-0.4cm,0.2cm>*{v_1}\restore\ar@{-}@(ul,ur)\ar@{-}[rrrr]&&&&*{\circ}\save[]+<0.5cm,0.2cm>*{v_n}\restore\ar@{-}@(ul,ur)\\
}}
\]
\caption{}\label{F2e}
\end{figure}
\begin{figure}
\[
\hbox{\xymatrix{
&&*{\circ}\save[]+<0.2cm,0.2cm>*{v_i}\restore
\ar@{-}[dll]\ar@{..}[drr]\ar@{-}[dd]&&\\
*{\circ}\save[]+<-0.4cm,0.2cm>*{v_1}\restore\ar@{-}[rrrr]\ar@{..}@(ul,ur)&&&&*{\circ}\save[]+<0.5cm,0.2cm>*{v_n}\restore\restore\ar@{-}@(ul,ur)\\
&&*{\circ}\save[]+<0.3cm,-0.2cm>*{v_j}\restore \ar@{..}[urr]&&}}
\hskip.5cm
\hbox{\xymatrix{
&\\
&\to\\
&}}
\hskip1cm
\hbox{\xymatrix{
&&*{\circ}\save[]+<0.2cm,0.2cm>*{v_i}\restore  \ar@{-}[dll]\ar@{-}[drr]  \ar@{..}[dd]&&\\
*{\circ}\save[]+<-0.4cm,0.2cm>*{v_1}\restore\ar@{-}[rrrr]\ar@{-}@(ul,ur)&&&&*{\circ}\save[]+<0.5cm,0.2cm>*{v_n}\restore\restore\ar@{..}@(ul,ur)\\
&&*{\circ}\save[]+<0.3cm,-0.2cm>*{v_j}\restore \ar@{-}[urr]&&}}
\]
\caption{}\label{F3e}
\end{figure}

It remains to show  that $\d\,''$ satisfies \eqref{E:loops}.  Define $m$ as follows: if the $d_i$ are all equal, put $m=n-1$, otherwise, define $m$ by the condition that $d_1=\dots=d_m$ and $d_m>d_{m+1}$. We have
$d''_i=d_i$ for all $i\not=m,n$, while $d''_m=d_m-1$ and $d''_n=d_n-1$. Consider condition \eqref{E:loops} for $\d\,''$:
\begin{equation}\label{E:loopdd}
\sum_{i=1}^k d''_i\leq  k^2+ \sum_{i=k+1}^n\min\{k,d''_i\} .  
\end{equation}
For $m\leq k<n$, we have
$\sum_{i=1}^k d''_i=\sum_{i=1}^k d_i-1$, while   $\sum_{i=k+1}^n\min\{k,d''_i\} \geq   \sum_{i=k+1}^n\min\{k,d_i\}-1$, and so \eqref{E:loopdd} holds.  For $k=n$,
$\sum_{i=1}^k d''_i=\sum_{i=1}^k d_i-2< k^2$, and so \eqref{E:loopdd} again holds. For $k<m$, first note that if $d_k\leq k$, then $\sum_{i=1}^k d''_i=\sum_{i=1}^k d_i\leq k^2\leq k^2+ \sum_{i=k+1}^n\min\{k,d''_i\}$.
So it remains to deal with the case where $k<m$ and $d_k> k$.  We have
\begin{equation*}
\sum_{i=1}^k d''_i=\sum_{i=1}^k d_i\leq  k^2+ \sum_{i=k+1}^n\min\{k,d_i\} .
\end{equation*}
Notice that as $d_i=d''_i$ except for $i=m,n$, we have  $\min\{k,d''_i\}=\min\{k,d_i\}$ except possibly for $i=m,n$. In fact, as $k<m$, we have  $d_m= d_k>k$ and $d''_m=d_m-1\geq k$ and so
$\min\{k,d_m\}=k=\min\{k,d''_m\}$. Hence $\sum_{i=k+1}^n\min\{k,d''_i\}\geq\sum_{i=k+1}^n\min\{k,d_i\}-1$. Thus, in order to establish  \eqref{E:loopdd}, it suffices to show that $\sum_{i=1}^k d_i<  k^2+ \sum_{i=k+1}^n\min\{k,d_i\}$. Suppose instead that $\sum_{i=1}^k d_i=  k^2+ \sum_{i=k+1}^n\min\{k,d_i\}$. We have 
 \[
 k d_m=\sum_{i=1}^k d_i=  k^2+ \sum_{i=k+1}^n\min\{k,d_i\}  
\]
and so
 \[
  d_m= k+\frac1k\sum_{i=k+1}^n\min\{k,d_i\} .  
\]
 Then
 \[
 \sum_{i=1}^{k +1}d_i =(k+1)d_m=k(k+1)+\frac{k+1}k\sum_{i=k+1}^n\min\{k,d_i\} .  
\]
We have $d_{k+1}=d_m>k$ and so
$\min\{k,d_{k+1}\}=k$. Note that $\sum_{i=k+2}^n\min\{k,d_i\}\not=0$ as $k+2\leq n$, since $k<m\leq n-1$. So 
 \[
 \sum_{i=1}^{k +1}d_i =k(k+1)+(k+1)+\frac{k+1}k\sum_{i=k+2}^n\min\{k,d_i\} >(k+1)^2+\sum_{i=k+2}^n\min\{k,d_i\},  
\]
contradicting \eqref{E:loops}. Hence $\d\,''$ satisfies \eqref{E:loopdd}, as claimed. This  completes the proof.
\qed

\emph{Acknowledgements}. This study began in 2010 during evening seminars on the mathematics of the internet conducted by the \emph{Q-Society}. The first author would like to thank the Q-Society members for their involvement, and particularly Marcel Jackson for his thought provoking questions. We also thank Yuri Nikolayevsky, whose comments improved the presentation of the paper.

\bibliographystyle{amsplain}
\bibliography{erdosgallai}

  \end{document}